\documentclass[12pt]{amsart}
\usepackage[margin=1.5in]{geometry}
\usepackage{amsfonts, amsmath , parskip}
\usepackage[T1]{fontenc}
\usepackage{mathrsfs, enumitem}
\usepackage{hyperref}
\usepackage[utf8]{inputenc}
\usepackage{amssymb}
\usepackage{amsfonts}
\usepackage{blindtext}
\usepackage{amsmath}
\usepackage{amsthm}
\usepackage{color}

\usepackage{hyperref}
\usepackage{csquotes}
\usepackage{tikz-cd}
\usepackage{graphicx}
\usepackage{pgfplots}
\pgfplotsset{compat=1.15}
\usepackage{mathrsfs}
\usetikzlibrary{arrows}
\allowdisplaybreaks[4]
\newtheorem{theorem}{Theorem}[section]
\newtheorem{lemma}[theorem]{Lemma}

\newtheorem{proposition}[theorem]{Proposition}

\theoremstyle{definition}
\newtheorem{definition}[theorem]{Definition}

\theoremstyle{remark}
\newtheorem{remark}[theorem]{Remark}
\theoremstyle{plain}
\newcommand{\ignore}[1]{}

\newcommand{\R}{\mathbb R}

\newcommand{\Z}{\mathbb Z}

\newcommand{\C}{\mathbb C}

\newcommand{\calB}{{\mathcal{B}}}
\newcommand{\calC}{{\mathcal{C}}}
\newcommand{\calD}{{\mathcal{D}}}

\newcommand{\calG}{{\mathcal{G}}}
\newcommand{\calH}{{\mathcal{H}}}

\newcommand{\calM}{{\mathcal{M}}}
\newcommand{\calN}{{\mathcal{N}}}

\newcommand{\calV}{{\mathcal{V}}}

\newcommand{\norm}[1]{\lVert #1 \rVert}
\DeclareMathOperator{\tr}{tr}

\DeclareMathOperator{\id}{id}
\title[Failure of the noncommutative Rota endpoint]{
Failure of noncommutative Rota's theorem with $L(\log L)^\alpha$ endpoint}
\author{Hiranmoy Ghorui}
\address{School of Mathematics, Indian Institute of Science Education and Research Thiruvananthapuram, Kerala 695551, India}
\email{hiran25@iisertvm.ac.in}

\author{Suman Mondal}
\address{School of Mathematics, Indian Institute of Science Education and Research Thiruvananthapuram, Kerala 695551, India}
\email{suman2024@iisertvm.ac.in}

\author{Samya Kumar Ray}
\address{The Institute of Mathematical Sciences, 4th Cross Street, CIT Campus, Tharamani, Chennai, Tamil Nadu 600113, India\\
Homi Bhabha National Institute, Training School Complex, Anushakti Nagar, Mumbai 400094, India}
\email{samya@imsc.res.in}

\begin{document}

\begin{abstract}
For every $1\le\alpha<2$, we construct a factorizable Markov map $T$ on a hyperfinite finite von Neumann algebra and a positive element $x\in L(\log L)^\alpha$ for which the sequence $T^n(T^*)^n x$ does not converge bilaterally almost uniformly. In particular, this answers a question of Anantharaman-Delaroche \cite{AD2006} concerning the $L\log L$ endpoint.
\end{abstract}

\maketitle

\section{Introduction and main result}

Let $(X,\mu)$ be a probability space and let $T$ be a
$\mu$-preserving Markov operator. The following classical convergence
theorem is due to G.-C.~Rota \cite{Rota1962}.
\begin{theorem}[Rota]
For every $p>1$ and $f\in L_p(X,\mu)$, the sequence
\[
    \bigl(T^n(T^*)^n f\bigr)_{n\geq 0}
\]
converges $\mu$-almost everywhere.
\end{theorem}
The same conclusion holds for every $f\in L\log L(X,\mu)$; see
\cite[Theorem~4]{Bufetov2002}. Anantharaman-Delaroche established a noncommutative analogue of Rota's
theorem for factorizable Markov maps on noncommutative $L_p$-spaces,
$p>1$, and raised the question of extending the theorem under minimal
hypotheses \cite[Corollary~4.2 and Theorem~6.5]{AD2006}. Hu \cite{Hu2009} extended this result to the $L\log^2 L(\calM,\tau)$ class of measurable operators. The purpose of this paper is to show that b.a.u. convergence can fail in every space $L(\log L)^\alpha(\calM,\tau)$, $1\le\alpha<2$,
even for factorizable Markov operator on a finite hyperfinite von Neumann algebra. Our example is partially motivated by Hong and Ricard \cite{HongRicard2024}.
\begin{theorem}\label{thm:main}
For every $1\le\alpha<2$, there exist a hyperfinite finite von Neumann algebra $(\calM,\tau)$,
a factorizable Markov map $T:\calM\to\calM$, and $x\ge0$ with
\begin{equation}\label{eq:main-modular}
\tau\bigl(x[\log(e+x)]^\alpha\bigr)\le\frac12
\end{equation}
such that, for every nonzero projection $e\in\calM$ and every $n_0\ge1$,
\begin{equation}\label{eq:main-unbounded}
\sup_{n\ge n_0}\norm{eT^n(T^*)^n(x)e}_\infty=\infty.
\end{equation}
In particular, $(T^n(T^*)^n(x))_{n\ge1}$ does not converge b.a.u.
\end{theorem}
We refer to Section~\ref{sec:preliminaries} for unexplained notations which are used in the introduction.

Let us now briefly discuss the ideas behind Theorem \ref{thm:main}.
\begin{itemize}
\item \emph{A finite-dimensional martingale construction.}
We construct a positive matrix $\mathbf{X}\in M_{2N}$,
conditional expectations $E_m$, and a conditional expectation $Q$
onto an abelian subalgebra such that
\[
\max_{m\in J_N}QE_m\mathbf{X}\ge c(\log N)^2
\]
on a set of normalized trace $1/8$, where $c>0$ is independent of $N$.

\item \emph{Lifting the martingale and realization by a single Markov map.}
We construct a factorizable Markov map $T_N$ on
$\bigoplus_{i=0}^{N+1}M_{2N}$ whose Rota sequence applied to
$z_N=(\mathbf{X},0,\ldots,0)$,  we get the identity
\[
T_N^{N-m+2}(T_N^*)^{N-m+2}(z_N)
=(QE_m\mathbf X,0,\ldots,0),\ 1\le m\le N.
\]
 The large values now occur on a set of trace
$1/(8(N+2))$. Taking $N+2$ independent tensor copies we construct a set
of trace at least $1-e^{-1/8}$. To this end we normalize by a factor $[\log(e+N(N+2))]^\alpha$ and obtain an element $Z_N$ which
satisfies
\[
\tau_{\calB_N}\!\left(
Z_N[\log(e+Z_N)]^\alpha
\right)\le\frac12
\] whereas on that set we obtain the lower bound, 
\[
\begin{aligned}
\max_{N/2\le n\le N+1} B_{N,n}(Z_N)
\geq\frac{c}{3^\alpha}(\log N)^{2-\alpha}.
\end{aligned}
\]
\item \emph{Infinite tensor product and divergence.}
In the third major step of our counterexample, we take an infinite tensor product of these finite-dimensional examples to obtain a factorizable Markov operator $T$ and form a
summable positive combination of finite-dimensional elements using the above step to obtain a special element $x\in L(\log L)^\alpha$ in an abelian subalgebra for which
\[
\limsup_{n\to\infty} T^n(T^\ast)^n(x)=\infty
\]
almost everywhere. By Lemma \ref{lem:compression} we then obtain
\[
\sup_{n\ge n_0}
\bigl\|eT^n(T^\ast)^n(x)e\bigr\|_\infty=\infty
\]
for every nonzero projection $e$ in the von Neumann algebra on which $T$ acts and every
$n_0\ge1$, proving Theorem \ref{thm:main}.
\end{itemize}
\smallskip
\noindent\textbf{Organization of the paper.} In Section~\ref{sec:preliminaries} we collect all necessary background and definitions. In Section~3 we establish an obstruction to b.a.u convergence via abelian sub-algebras. In Sections~4 and 5 we build the essential finite-dimensional examples, and Section~6 glues these blocks to prove Theorem~\ref{thm:main}. In the final sections we compare our results with recent pre-prints where the authors claimed to extend noncommutative-$L_p$-Rota theorem in Orlicz spaces.

\section{Preliminaries}\label{sec:preliminaries}
Throughout the paper, $(\calM,\tau)$ is a finite von Neumann algebra with a faithful normal tracial state. A closed, densely defined operator $a$ affiliated with $\mathcal{M}$ is said to be $\tau$-measurable if, for
every $\delta>0$, there exists a projection $e\in\mathcal{P}(\mathcal{M})$
such that $\tau(e^\perp)\le\delta$ and $\|ae\|_\infty<\infty$.
Let $L_0(\mathcal{M})$ denote the collection of all $\tau$-measurable
operators associated with $(\mathcal{M},\tau)$. For $1\le\alpha<2$, put
\begin{equation}\label{eq:OrlicZ}
\Phi_\alpha(t)=t[\log(e+t)]^\alpha,\qquad t\ge0.
\end{equation}
We denote by $L(\log L)^\alpha(\calM,\tau)$ the space of all $\tau$-measurable operators $x$ affiliated with $\calM$ satisfying $\tau(\Phi_\alpha(|x|))<\infty$. The function $\Phi_\alpha$ is increasing and convex, with $\Phi_\alpha(0)=0$. We write $M_n$ for the set of all $n\times n$ complex matrices equipped with normalized trace denoted by $\tr_n$.
\begin{definition}
  A Markov map on $(\mathcal M,\tau)$ is a normal unital completely
positive map $T:\calM\to\calM$ satisfying $\tau\circ T=\tau$.  
\end{definition}
Let $T:\mathcal M\to\mathcal M$ be a Markov map. There exists a unique normal completely positive map $T^*:\mathcal M\to\mathcal M$ characterized by
\[
\tau(T(a)b)=\tau(aT^*(b)),\qquad a,b\in\calM.
\]
Clearly, $T^*$ is again a Markov map on $\mathcal M$ and is called the adjoint of $T.$ Moreover, both $T$ and $T^*$ extend to positive contractions on $L_p(\calM,\tau)$ for all $1\leq p\leq\infty.$ 
\begin{definition}A Markov map $T:\mathcal M\to\mathcal M$ is called factorizable if there exists a finite
tracial von Neumann algebra $\calN$ and normal unital trace-preserving
$*$-homomorphisms $j_0,j_1:\calM\to\calN$ such that
\[
T=j_0^*j_1.
\]
\end{definition}
This is the formulation of factorizability in the tracial setting introduced in
\cite[Definition~6.2]{AD2006}; see also
\cite[Definition~1.3 and Remark~1.4(a)]{HM11}.
We require the following property of factorizable Markov maps which maybe well-known to the experts but we give proof for the shake of completeness.
\begin{proposition}\label{prop:permanence}
Let $(\calM_k,\tau_k)$ be finite von Neumann algebras with faithful
normal tracial states, and let $S_k:\calM_k\to\calM_k$ be factorizable
Markov maps. On the tracial infinite tensor product
$\calM=\overline{\bigotimes}_{k\ge1}\calM_k$, there is a factorizable
Markov map $T$ satisfying
\begin{equation}\label{eq:tensor-map}
\begin{aligned}
&T(a_1\otimes\cdots\otimes a_m\otimes1\otimes\cdots)\\
&\qquad=S_1(a_1)\otimes\cdots\otimes S_m(a_m)\otimes1\otimes\cdots
\end{aligned}
\end{equation}
for all $m\ge1$ and $a_k\in\calM_k$, $1\le k\le m$.
Its trace adjoint satisfies
\[
\begin{aligned}
&T^*(a_1\otimes\cdots\otimes a_m\otimes1\otimes\cdots)\\
&\qquad=S_1^*(a_1)\otimes\cdots\otimes S_m^*(a_m)\otimes1\otimes\cdots.
\end{aligned}
\]
The same assertion holds for finite tensor products.
\end{proposition}
\begin{proof}
We choose tracial factorizations $S_k=j_{0,k}^*j_{1,k}$ through
$(\calN_k,\sigma_k)$.
The following maps
\[
j_i=\overline{\bigotimes}_{k\ge1}j_{i,k}:
\calM\longrightarrow\overline{\bigotimes}_{k\ge1}\calN_k,
\qquad i=0,1,
\]
are normal unital trace-preserving $*$-homomorphisms.
Thus $T=j_0^*j_1$ is factorizable. This completes the proof of the proposition.
\end{proof}
We refer the interested readers to \cite{AD2006},\cite{Ric08},\cite{Haagerup2011},\cite{KUMMERER1985} papers on noncommutative Markov maps.
For noncommutative analogues of almost everywhere convergence, we refer to
Lance \cite{Lance1976} and Jajte \cite{Jajte1985}.
\begin{definition}
 Let $(y_n)$ be a sequence of $\tau$-measurable operators. We say $(y_n)$ converges bilaterally almost uniformly, abbreviated b.a.u., to $y$ if, for every
$\varepsilon>0$, there is a projection $e\in\calM$ such that
\[
\tau(1-e)<\varepsilon,\qquad
\norm{e(y_n-y)e}_\infty\longrightarrow0.
\]   
\end{definition}
\begin{definition}
Let $\mathcal N$ be a von Neumann subalgebra of $\mathcal M.$ Then there exists a unique trace-preserving normal contractive completely positive projection $\mathcal E:\mathcal M\to\mathcal N$ satisfying $\mathcal{E}(axb)=a\mathcal{E}(x)b$ for all $a,b\in\mathcal N$ and $x\in\mathcal M,$ which is called the conditional expectation of $\mathcal M$ on $\mathcal N.$ 
\end{definition} 
We refer the readers to this for more on \cite{Takesaki1979}.
\section{An obstruction to b.a.u. convergence via abelian sub-algebra}
We first record an elementary lemma which allows us to lift divergence result from an abelian von Neumann subalgebra to the whole von Neumann algebra. Our argument is inspired by the proof of \cite[Lemma~4.5]{CXY13}.
\begin{lemma}\label{lem:compression}
Let $\calG\subseteq\calM$ be a unital abelian von Neumann subalgebra, and let
$(y_n)$ be a sequence in $L_1(\calG)_+$. Suppose that
\begin{equation}\label{eq:scalar-divergence}
\limsup_{n\to\infty}y_n=\infty
\qquad\text{almost everywhere in }\calG.
\end{equation}
Then, for every nonzero projection $e\in\calM$ and every $n_0\ge1$,
\[
\sup_{n\ge n_0}\norm{ey_ne}_\infty=\infty.
\]
In particular, $(y_n)$ does not converge b.a.u. in $\calM$.
\end{lemma}

\begin{proof}
Suppose that $ey_ne\le Ke$ for $n\ge n_0$, where $K<\infty$.
Let $E_{\calG}:\calM\to\calG$ be the trace-preserving conditional expectation
and put $b=E_{\calG}(e)$. For $r>0$, the bounded truncation
$y_{n,r}=\min(y_n,r)$ belongs to $\calG$ and satisfies $ey_{n,r}e\le Ke$.
Applying the Schwarz inequality for the completely positive map $E_{\calG}$
to $y_{n,r}^{1/2}e$, and using bimodule property of conditional expectation and for $n\geq n_0$ we obtain
\[
\begin{aligned}
b^2y_{n,r}
&=E_{\calG}(y_{n,r}^{1/2}e)^*E_{\calG}(y_{n,r}^{1/2}e)\\
&\le E_{\calG}(ey_{n,r}e)\le Kb.
\end{aligned}
\]
Let $r\to\infty$. On $\{b>0\}$ this implies
$y_n\le K/b$ for every $n\ge n_0$, contradicting
\eqref{eq:scalar-divergence} unless $b=0$. Thus $\tau(e)=\tau(b)=0$,
and faithfulness gives $e=0$.

For the final assertion, suppose $y_n\to y$ b.a.u. Choose a projection $e_1$
with $\tau(1-e_1)<1/4$ such that
$\norm{e_1(y_n-y)e_1}_\infty\to0$. Since $y$ is measurable, by definition
there exists a projection $q$ such that $\tau(1-q)<1/4$ and $R:=\norm{yq}_{\infty}<\infty$.
Then $e=e_1\wedge q$ is nonzero and $eye$ is bounded. Indeed one observes
\begin{align*}
 \tau(1-e)&\leq \tau(1-e_1)+\tau(1-q)\\
 &\leq \frac{1}{4}+\frac{1}{4}\\
 &\leq \frac{1}{2}
\end{align*}
So, $\tau(e)\geq \frac{1}{2}$ and $e$ is non zero. Now $\norm{eye}_{\infty}=\norm{e(yq)e}_{\infty}\leq R$. As $yq$ is bounded then $eye$ is also bounded.
However, for all $n\geq n_0$
\begin{align*}
  \norm{ey_ne}_{\infty}&\leq \norm{e(y_n-y)e}_{\infty} +\norm{eye}_{\infty}\\
  &\leq \norm{e(e_1(y_n-y)e_1)e}_{\infty}+\norm{eye}_{\infty}\\
  &< 1+R<\infty.
\end{align*}
In the last inequality we use $\norm{e_1(y_n-y)e_1}_\infty\to0$ and which gives a contradiction.
\end{proof}

\section{Construction of a martingale and diagonal estimate}
Every matrix algebra in this section is equipped with its normalized trace.
Fix an integer $N\ge16$ divisible by four. Let
$
\eta=(1,\ldots,1)^t\in\R^N
$
and define the real matrix $H\in M_N$ by
\begin{equation}\label{eq:hilbert-matrix}
H_{jk}=\begin{cases}
\dfrac{1}{\pi(j-k)},&j\ne k,\\[4pt]
0,&j=k.
\end{cases}
\end{equation}
This matrix plays an important role in noncommutative martingale theory. We refer  \cite{jungexu2005} on this.
The norm estimate below can be obtained from \cite[Theorem~2, inequality~(1.6)]{MontgomeryVaughan1974} and \cite[Lemma~2.4]{Xia2023}; we include a proof for the reader's convenience.
\begin{lemma}\label{lem:H-contraction}
The matrix $H$ in \eqref{eq:hilbert-matrix} satisfies $\norm{H}\le1$.
\end{lemma}
\begin{proof}
The function
\[
m(\theta)=\frac{i(\pi-\theta)}{\pi},\qquad 0<\theta<2\pi,
\]
has essential supremum one. With the convention
\[
\widehat m(r)=\frac1{2\pi}\int_0^{2\pi}m(\theta)e^{-ir\theta}\,d\theta,
\]
integration by parts gives $\widehat m(0)=0$ and
$\widehat m(r)=1/(\pi r)$ for $r\ne0$. Note that $T_m:L_2(\mathbb T)\to L_2(\mathbb T)$ defined by $T_m(f):=mf$ has operator norm $\|m\|_\infty=1$. Let $\calH=\operatorname{span}\{e^{i\theta},\ldots,e^{iN\theta}\}$ and $\calB=(e^{i\theta},\ldots,e^{iN\theta})$. If $P:L_2(\mathbb T)\to\calH$ is the orthogonal projection and $i:\calH\to L_2(\mathbb T)$ is the inclusion, then $PT_mi$ is a contraction whose matrix representation with respect to $\calB$ is $H$.
\end{proof}
Put
\[
R=H(1-\frac{\eta\eta^*}{N}),\qquad W=(1-RR^*)^{1/2}.
\]
Both matrices are real, $\norm{R}\le1$, and $R\eta=0$.
Regard the rows of the block matrix $[R\ W]\in M_{N,2N}$ as vectors $v_1,\ldots,v_N\in\C^{2N}$.
They are orthonormal because $RR^*+WW^*=1$. They are also orthogonal to
\[
v_0=\frac{1}{\sqrt{N}}(\eta,0)^t.
\]
Complete $v_0,v_1,\ldots,v_N$ to an ordered orthonormal basis $\calV_N$
of $\C^{2N}$, where
\[
\calV_N=(v_0,v_1,\ldots,v_{2N-1}).
\]Let $V=(v_0\ v_1\ \cdots\ v_{2N-1})$ be the unitary matrix whose columns
are these basis vectors, and let $\Delta:M_{2N}\to\calD_{2N}$ be the conditional expectation onto
the diagonal entries in the usual basis, where $\mathcal D_{2N}$ is the diagonal subalgebra of $M_{2N}.$
Let $\calD_{\calV_N}=V\calD_{2N}V^*$ be a subalgebra of $M_{2N}.$ Let
$Q:M_{2N}\to\calD_{\calV_N}$ be the trace-preserving conditional expectation given by \begin{equation}\label{defnofQ}
Q(A):=V\Delta(V^*AV)V^*,\qquad A\in M_{2N}.
\end{equation}
 Set
\begin{equation}\label{eq:X}
\mathbf{X}=Nv_0v_0^*=\begin{pmatrix}\eta\eta^*&0\\0&0\end{pmatrix}.
\end{equation}
In particular,
\begin{equation}\label{eq:X-properties}
Q\mathbf{X}=\mathbf{X},\qquad \tr_{2N}(\mathbf{X})=\frac12,\qquad \norm{\mathbf{X}}_\infty=N.
\end{equation}
For $1\le m\le N$, let $E_m$ be the trace-preserving conditional expectation onto
\[
M_m\oplus\C^{2N-m}\subseteq M_{2N}
\]
in the original coordinate basis. More precisely,
\begin{equation}\label{mthconditionalexpec}
E_m(A)=\begin{pmatrix}(a_{ij})_{1\le i,j\le m} & 0 \\
0 & \operatorname{diag}(a_{m+1,m+1},\ldots,a_{2N,2N})
\end{pmatrix}.
\end{equation}
In particular,
\begin{equation}\label{eq:nesting}
E_mE_\ell=E_{\min(m,\ell)}.
\end{equation}
Throughout the rest of the article we will use the above notations.
\begin{proposition}\label{prop:matrix-peaks}
Set
\[
J_N=\{N/4,N/4+1,\ldots,N/2-1\},\qquad c=\frac1{4\pi^2}.
\]
Then for every  $j\in J_N$,
\begin{equation}\label{eq:matrix-peak}
QE_j\mathbf{X}\ge c(\log N)^2v_jv_j^*.
\end{equation}
Moreover, $\tr_{2N}\bigl(\sum_{j\in J_N}v_jv_j^*\bigr)=1/8$.
\end{proposition}
\begin{proof}
Let $\eta_m\in\C^{2N}$ have its first $m$ entries equal to one and all remaining
entries zero. From \eqref{eq:X},
\[
E_m\mathbf{X}=\eta_m\eta_m^*+\sum_{k=m+1}^N e_{kk}.
\]
Since the first $N$ coordinates of $v_j$ are the entries of row $j$ of $R$,
\begin{equation}\label{eq:coefficient}
\langle v_j,E_j\mathbf{X}v_j\rangle
=\left|\sum_{k=1}^jR_{jk}\right|^2
+\sum_{k=j+1}^N|R_{jk}|^2
\ge\left|\sum_{k=1}^jR_{jk}\right|^2.
\end{equation}
Write $h_r=\sum_{\ell=1}^r\ell^{-1}$ denote the \(r\)-th partial sum of the harmonic series, with $h_0=0$.
By the definition of $H$ we have $H_{jj}=0$ and for $j\neq k$,  $H_{jk}=\frac{1}{\pi(j-k)}$.
So,
\begin{align*}
    \sum_{k=1}^{N}H_{jk}&=\sum_{k=1}^{j-1}H_{jk}+\sum_{k=j+1}^{N}H_{jk}\\
    &=\frac{1}{\pi}\sum_{k=1}^{j-1}\frac{1}{j-k}+\frac{1}{\pi}\sum_{k=j+1}^{N}\frac{1}{j-k}\\
    &=\frac{h_{j-1}}{\pi}-\frac{1}{\pi}\sum_{k=j+1}^{N}\frac{1}{k-j}\\
    &=\frac{h_{j-1}}{\pi}-\frac{h_{N-j}}{\pi}
\end{align*}
Thus $j$th row-sum of $H$ is $(h_{j-1}-h_{N-j})/\pi$.
Since $R=H-H\frac{\eta\eta^*}{N}$,
\begin{equation}\label{eq:row-sum}
\begin{aligned}
\sum_{k=1}^jR_{jk}
&=\frac{h_{j-1}}{\pi}-\frac{j}{N\pi}(h_{j-1}-h_{N-j})\\
&=\frac1\pi\left[\left(1-\frac jN\right)h_{j-1}
+\frac jN h_{N-j}\right].
\end{aligned}
\end{equation}
For $j\in J_N$,  we get $j-1\geq \frac{N}{4}-1$  and $N-j\geq \frac{N}{4}-1$. Then both  $(j-1)$-th and $(N-j)$-th partial sum of the harmonic series in the last expression are at least
\[
h_{N/4-1}\ge\log(N/4)\ge\frac12\log N.
\]
Equations \eqref{eq:coefficient} and \eqref{eq:row-sum} give the lower bound,
$\langle v_j,E_j\mathbf{X}v_j\rangle\ge c(\log N)^2$. As $QE_j\mathbf{X}$ is positive and diagonal
in the basis $\calV_N$ containing $v_j$, the lower bound for its \(j\)-th diagonal entry gives exactly \eqref{eq:matrix-peak}. Finally,  as $|J_N|=N/4$ and each $v_jv_j^*$ has normalized trace $\frac{1}{2N}$ the rest of the proposition follows.
\end{proof}

\section{Finite-dimensional Markov maps lifting the martingale}
\subsection{The cyclic coordinate}
Let $\calC_N=\ell_\infty(\Z/(N+2)\Z)\otimes M_{2N}$
with the tensor trace
$\tau_{\calC_N}((f_i))=(N+2)^{-1}\sum_i\tr_{2N}(f_i)$. Define
\[
\phi_0=Q,\qquad
\phi_i=E_{N-i+1}\quad(1\le i\le N),\qquad
\phi_{N+1}=\id,
\]
 Where $Q$ and $E_{N-i+1}$ are taken from \eqref{defnofQ},\eqref{mthconditionalexpec} and put
\begin{equation}\label{eq:TN}
(T_Nf)_i=\phi_i(f_{i+1}),\qquad i\in\Z/(N+2)\Z.
\end{equation}
The map $T_N$ is unital, completely positive and trace preserving.
Since every $\phi_i$ is self-adjoint with respect to the usual trace pairing,
\begin{equation}\label{eq:TN-adjoint}
(T_N^*f)_i=\phi_{i-1}(f_{i-1}).
\end{equation}
Let $z_N=(\mathbf{X},0,0,....,0)\in\calC_N$ and by the definition of trace on $\calC_N$ we obtain,
\begin{align}\label{eq:z_N}
    \tau_{\calC_N}(z_N)&=\frac{1}{N+2}\tr_{2N}(X)=\frac{1}{2(N+2)}
\end{align}
\begin{proposition}\cite[Example 6.3.(c)]{AD2006}\label{pro:factorizable}
Let $I$ be a finite nonempty set, and let $(A,\psi)$ be a von Neumann algebra with a faithful normal state and let $q=(q_{ij})_{i,j\in I}$
be a transition matrix with a strictly positive stationary probability
distribution $\pi=(\pi_i)_{i\in I}$. Thus
\[
q_{ij}\geq 0,\qquad\sum_{j\in I}q_{ij}=1,\qquad\pi_i>0,\qquad\sum_{i\in I}\pi_i=1,
\qquad\sum_{i\in I}\pi_iq_{ij}=\pi_j.
\]
Suppose that each $P_i:(A,\psi)\to(A,\psi)$ is factorizable. On the von Neumann algebra
\[
B=\bigoplus_{i\in I}A,
\]
equipped with
\[
\varphi(b)=\sum_{i\in I}\pi_i\psi(b_i), \qquad b=(b_i)_{i\in I}\in B,
\]
define
\[
P(b)_i=\sum_{j\in I}q_{ij}P_i(b_j),
\qquad i\in I.
\]
Then $P$ is a factorizable $\varphi$-preserving Markov map.
\end{proposition}
\begin{lemma}\label{lem:cyclic}
The map $T_N$ is factorizable.
If $A_{N,n}=T_N^n(T_N^*)^n$, then
\begin{equation}\label{eq:realization}
(A_{N,n}z_N)_i =
\begin{cases}
QE_{N-n+2}\mathbf{X}, & \text{if } i=0,\\
0, & \text{if } i\neq 0,
\end{cases}
\qquad 2\le n\le N+1.
\end{equation}
All the operators $z_N,A_{N,1}z_N,A_{N,2}z_N,\ldots$ belong to $
\ell_\infty(\Z/(N+2)\Z)\otimes\calD_{\calV_N}.$
\end{lemma}
\begin{proof}
By \cite{Ric08} every $\phi_i$ is factorizable. Now, we define, $I=\Z/(N+2)\Z$ and  $\pi_i=\frac{1}{N+2}$  ,  \[
{q_{ij}=
\begin{cases}
1, & j=i+1 \\
0, & \text{otherwise}.
\end{cases}
}
\]
where $i,j\in I $. Then we obtain, 
\[
q_{ij}\geq 0,\qquad\sum_{j\in I}q_{ij}=1,\qquad\pi_i>0,\qquad\sum_{i\in I}\pi_i=1,
\qquad\sum_{i\in I}\pi_iq_{ij}=\frac{1}{N+2}=\pi_j.
\]
 Thus by \ref{eq:TN}, we have
 \[
 (T_Nf)_i=\phi_i(f_{i+1})=\sum_{j\in I}q_{ij}\phi_i(f_j),\qquad i\in I.
 \]
Now using the proposition \ref{pro:factorizable} on $T_N$ we get, $T_N$ is factorizable.\\
Let $f\in\mathcal{C}_N$. Applying \eqref{eq:TN} and \eqref{eq:TN-adjoint} repetitively, we obtain for every $n\ge1$,
\begin{equation}\label{eq:palindrome}
(A_{N,n}f)_i
=\phi_i\phi_{i+1}\cdots\phi_{i+n-1}
\phi_{i+n-1}\cdots\phi_{i+1}\phi_i(f_i),
\end{equation}
with all indices read modulo $(N+2)$. Note that it follows from \eqref{eq:palindrome} that $(A_{N,n}z_N)_i=0$ whenever $i\neq 0.$ 
For $2\le n\le N+1$, we use the martingale property \eqref{eq:nesting} in \eqref{eq:palindrome} to obtain
\[
(A_{N,n}z_N)_0=QE_{N-n+2}Q\mathbf{X}=QE_{N-n+2}\mathbf{X}.
\]
{Also, if we put $n=1$ in \eqref{eq:palindrome} then, 
\[
(A_{N,1}z_N)_{0}=\phi_{0}\phi_{0}\mathbf{X}=Q\mathbf{X}=\mathbf{X}, \qquad  (A_{N,1}z_N)_{i}=0 \text{ for }, i\neq 0  .
\]}
Moreover, as range of $Q$ is $\mathcal{D}_{\mathcal{V}_N}$ we see that $z_{N}, A_{N,1}z_N,...\in \ell_\infty(\Z/(N+2)\Z)\otimes\calD_{\calV_N}.$ This completes the proof.
\end{proof}

\subsection{Amplification of the exceptional set}
Let us set
\begin{equation}\label{eq:amplification-def}
\calB_N=\calC_N^{\otimes (N+2)},\qquad
S_N=T_N^{\otimes (N+2)},\qquad
Y_N=\sum_{a=1}^{N+2} z_N^{(a)}.
\end{equation}
Here for any $w$ we define
\[
w^{(a)}:=1\otimes\cdots\otimes\underbrace{w}_{a\text{-th factor}}
\otimes\cdots\otimes1
\]
Also, $z_N^{(a)}$ follows the above definition.\\
Let
\[
B_{N,n}=S_N^n(S_N^*)^n,\qquad \calG_N=(\ell_\infty(\Z/(N+2)\Z)\otimes\calD_{\calV_N})^{\otimes (N+2)}.
\]
By Proposition~\ref{prop:permanence}, $S_N$ is factorizable. Unitality gives
\begin{equation}\label{eq:amplified-images}
B_{N,n}Y_N=\sum_{a=1}^{N+2}(A_{N,n}z_N)^{(a)}\in(\calG_N)_+.
\end{equation}
\begin{proposition}\label{prop:amplification}
There is a projection $F_N\in\calG_N$ such that
\begin{align}
\tau_{\calB_N}(Y_N)&=\frac12,
&0\le Y_N&\le (N+2)N\,1,\label{eq:Y-properties}\\
\tau_{\calB_N}(F_N)&\ge\beta:=1-e^{-1/8},\label{eq:F-trace}\\
\max_{N/2\le n\le N+1}B_{N,n}Y_N
&\ge c(\log N)^2\quad\text{on }F_N.\label{eq:amplified-peak}
\end{align}
The maximum and the last inequality are taken in the abelian algebra
$\calG_N$.
Moreover, for $1\le\alpha<2$, if
\[
q_N=\log(e+(N+2)N),\qquad Z_N=Y_N/q_N^\alpha,
\]
then
\begin{equation}\label{eq:Z-modular}
\tau_{\calB_N}\bigl(\Phi_\alpha(Z_N)\bigr)\le\frac12,
\qquad q_N\le3\log N.
\end{equation}
\end{proposition}

\begin{proof}
\textbf{Checking} \ref{eq:Y-properties}: By \eqref{eq:z_N}, $\tau_{\calC_N}(z_N)=1/(2(N+2))$ and $0\le z_N\le N1$.
This proves, 
\begin{align*}
    \tau_{\calB_N}(Y_N)&=\tau_{\calB_N}(\sum_{a=1}^{N+2} z_N^{(a)})
    =\sum_{a=1}^{N+2}\tau_{\calB_N}( z_N^{(a)})
    =\sum_{a=1}^{N+2}\tau_{\calC_N}(z_N)=\frac{1}{2}.
\end{align*}
and $0\leq Y_N\leq \sum_{a=1}^{N+2}N1=(N+2)N1.$\\
\textbf{Checking} \ref{eq:F-trace}:
In factor $a$, let $F_{N,a}$ be the copy of
the projection with zeroth coordinate $\sum_{j\in J_N}v_jv_{j}^*$ and all other
coordinates zero i.e,
\[
F_{N,a}=1\otimes\cdots\otimes\underbrace{(\sum_{j\in J_N}v_jv_{j}^*,0,\cdots,0)}_{a\text{-th factor}}
\otimes\cdots\otimes1
\]
The $(N+2)$ projections $F_{N,a}$ are independent in $\calG_N$,
and each has trace $1/(8(N+2))$. Thus, for $F_N=\bigvee_{a=1}^{(N+2)}F_{N,a}$,  and as $F_{N,a}$ are commuting projection then, 
\[ 
F_{N}^{\perp}=\bigwedge_{a=1}^{N+2}F_{N,a}^{\perp}=\prod_{a=1}^{N+2}F_{N,a}^{\perp}
\]
and we get,
\begin{align*}
\tau_{\calB_N}(F_N)&=1-\prod_{a=1}^{N+2}\tau_{\calB_N}(F_{N,a}^{\perp})
=1-\left(1-\frac1{8(N+2)}\right)^{(N+2)}\ge1-e^{-1/8}.
\end{align*}
\textbf{Checking} \ref{eq:amplified-peak}:
{For $j\in J_N$, we define $p_j = (v_j v_j^*, 0, \ldots, 0) \in \calC_N$ and $n(j)=N-j+2$ which lies between $\frac{N}{2}$ and $N+1$. By the  Lemma \ref{lem:cyclic}, $A_{N,n(j)}z_N = (Q E_j X, 0, \ldots, 0)$. Therefore, using Proposition \ref{prop:matrix-peaks} we get,
\[
A_{N,n(j)}z_N\ge c(\log N)^2p_j.
\] 
From the identity \eqref{eq:amplified-images}, all the operators in the sum  in $B_{N,n(j)}Y_N$ are positive. Thus for each $1\leq a\leq N+2$, 
\[
B_{N,n(j)}Y_N
=\sum_{r=1}^{N+2}(A_{N,n(j)}z_N)^{(r)}
\ge (A_{N,n(j)}z_N)^{(a)}
\ge c(\log N)^2p_j^{(a)}.
\]
In the last inequality we use $w\mapsto w^{(a)}$  is a $*$-homomorphism.
Since $\frac{N}{2}\leq n(j)\leq N+1$, the maximum taken in $\calG_N$ satisfies
\[
\max_{N/2 \le n \le N+1} B_{N,n}Y_N
\ge B_{N,n(j)}Y_N
\ge c(\log N)^2 p_j^{(a)}
\]
for every $1 \le a \le N+2$ and $j \in J_N$.
All these operator lie in $\calG_N$, where 
\[
\bigvee_{a=1}^{N+2}\ \bigvee_{j \in J_N} p_j^{(a)}
= \bigvee_{a=1}^{N+2} F_{N,a}
= F_N.
\]
Thus in the abelian $\calG_N$, join of these projections $p_j^{(a)}$ for every $1 \le a \le N+2$ and $j \in J_N$ we obtain,
\[
\max_{N/2 \le n \le N+1} B_{N,n}Y_N
\ge c(\log N)^2
    \bigvee_{a=1}^{N+2}\bigvee_{j \in J_N} p_j^{(a)}
= c(\log N)^2 F_N.
\]
This proves \eqref{eq:amplified-peak}.}\\
\textbf{Checking} \eqref{eq:Z-modular}:
Finally, $q_N\ge1$ and $0\le Y_N\le (N+2)N1$. Therefore, as $\Phi$ is increasing
\[
\begin{aligned}
\Phi_\alpha(Z_N)
&=\frac{Y_N}{q_N^\alpha}
\left[\log\left(e+\frac{Y_N}{q_N^\alpha}\right)\right]^\alpha\\
&\le\frac{Y_N}{q_N^\alpha}
[\log(e+(N+2)N)]^\alpha=Y_N.
\end{aligned}
\]
which proves \eqref{eq:Z-modular}. Since $e+N(N+2)\le N^3$ for $N\ge16$,
the remaining estimate follows.
\end{proof}

\section{Gluing the finite-dimensional examples via tensor product}

\begin{proof}[Proof of Theorem~\ref{thm:main}]
Fix $1\le\alpha<2$. For $k\ge1$, choose
\[
m_k=\left\lceil2^{\,2k/(2-\alpha)}\right\rceil,
\qquad N_k=2^{m_k}.
\]
Each $N_k$ is divisible by four and at least sixteen.
Form the tracial infinite tensor products
\begin{equation}\label{eq:infinite-algebras}
(\calM,\tau)=\overline{\bigotimes}_{k\ge1}(\calB_{N_k},\tau_{\calB_{N_k}}),
\qquad
\calG=\overline{\bigotimes}_{k\ge1}\calG_{N_k}\subseteq\calM.
\end{equation}
The algebra $\calM$ is finite and hyperfinite, because the increasing  finite tensor products are finite-dimensional and their union is weakly dense. The algebra $\calG$ is abelian. Let
\[
T=\bigotimes_{k\ge1}S_{N_k}
\]
be the factorizable Markov map described in Proposition~\ref{prop:permanence}.\\
{For the canonical embedding we define,
\[
\iota_k(a)=1^{\otimes(k-1)}\otimes a\otimes1\otimes\cdots,\qquad a\in \mathcal B_{N_k}.
\]
Let, $P_{N_k}$ be any operator on $\calB_{N_k}$, where $\calB_{N_k}$ is define from \eqref{eq:amplification-def}. Then without any ambiguity we identify the operator  $P_{N_k}$ as, the image under $\iota_k$ in $\calM$, i.e, 
$P_{N_k}=\iota_k(P_{N_k})\in \calM$. For the rest of the paper, the operators $Y_{N_k},Z_{N_k}$ and $F_{N_k}$ are following  the above definition.}

Define,
\begin{equation}\label{eq:input}
x=\sum_{k\ge1}2^{-k}Z_{N_k}
=\sum_{k\ge1}\frac{2^{-k}}{q_{N_k}^{\alpha}}Y_{N_k}.
\end{equation}
Clearly
\eqref{eq:Y-properties}, the above series defines a positive element of
$L_1(\calG)$. Then by convexity of $\Phi_\alpha$ and
\eqref{eq:Z-modular} give
\[
\tau\!\left[\Phi_\alpha\!\left(\sum_{k=1}^m2^{-k}Z_{N_k}\right)\right]
\le\sum_{k=1}^m2^{-k}\tau\bigl(\Phi_\alpha(Z_{N_k})\bigr)
\le\frac12.
\]
By the monotone convergence theorem we obtain
\eqref{eq:main-modular}.
Let $A_n=T^n(T^*)^n$. Note that
\begin{equation}\label{eq:infinite-images}
A_nx=\sum_{k\ge1}\frac{2^{-k}}{q_{N_k}^\alpha}B_{N_k,n}Y_{N_k}
\in L_1(\calG)_+.
\end{equation}
On $F_{N_k}$, equations \eqref{eq:amplified-peak}, \eqref{eq:Z-modular} and
\eqref{eq:infinite-images} give
\begin{equation}\label{eq:divergent-peaks}
\begin{aligned}
\max_{N_k/2\le n\le N_k+1}A_nx
&\ge\frac{c\,2^{-k}(\log N_k)^2}{q_{N_k}^\alpha}\\
&\ge\frac{c}{3^\alpha}\,2^{-k}(\log N_k)^{2-\alpha}\\
&\ge\frac{c(\log2)^{2-\alpha}}{3^\alpha}\,2^k.
\end{aligned}
\end{equation}
The projections $F_{N_k}$ clearly are
independent in $\calG$ if we identify $\calG$ as a classical measure space. Each has trace at least $\beta>0$. Therefore, by Borel-Cantelli lemma almost every point in the underlying measure space of $\mathcal G$ belongs to infinitely many $F_{N_k}$. Therefore, from
\eqref{eq:divergent-peaks} we obtain
\[
\limsup_{n\to\infty}A_nx=\infty
\qquad\text{almost everywhere in }\calG.
\]
Lemma~\ref{lem:compression} proves \eqref{eq:main-unbounded} and the failure
of b.a.u. convergence.
\end{proof}

\section{The endpoint claims in a recent preprints}
In this section first we discuss the Orlicz spaces. For more details the reader can be found in \cite{book}.\\
Let $\Phi\colon [0,\infty)\to[0,\infty)$ be an Orlicz function,
i.e., a continuous, increasing, and convex function satisfying
$\Phi(0)=0$ and $\lim_{t\to\infty}\Phi(t)=\infty$.
\begin{definition}
 Given a semifinite von Neumann algebra $(\mathcal{M},\tau)$ and an
Orlicz function $\Phi$, the noncommutative Orlicz space
$L_\Phi(\mathcal{M},\tau)$ is defined by
\[
L_\Phi(\mathcal{M},\tau)=\left\{x\in L_0(\mathcal{M},\tau):
\tau\!\left(\Phi\!\left(\lambda|x|\right)\right)<\infty
\text{ for some } \lambda>0\right\}.
\]
The space $L_\Phi(\mathcal{M},\tau)$, equipped with the Luxemburg norm
\[
\|x\|_\Phi=\inf\left\{\lambda>0:\tau\!\left(\Phi\!\left(\frac{|x|}{\lambda}\right)\right)\le 1
\right\},\qquad x\in L_\Phi(\calM),
\]
is a Banach space.   
\end{definition}
\begin{definition}($\Delta_2$-condition). 
  An Orlicz function $\Phi\colon [0,\infty)\to[0,\infty)$ is said to
satisfy the $\Delta_2$-condition if there exists a constant $k>0$
such that
\[
\Phi(2t)\le k\Phi(t), \qquad t\ge 0.
\]  
We remark that if $\Phi\in \Delta_2$ then for every $x\in L_0(\calM)$, $\tau(\Phi(|x|))<\infty$ if and only if $x\in L_\Phi(\calM)$. Also we note that the function $\Phi_s(t)=t\bigl(\log(1+t)\bigr)^s$ for $s>0$ is an Orlicz function and satisfies the $\Delta_2$ condition . Similarly \eqref{eq:OrlicZ} satisfies the $\Delta_2$-condition and the given definition of $L(\log L)^\alpha(\calM,\tau)$ is coincide with $L_\Phi(\calM,\tau).$ Note that for the rest of the paper we don't assume the $\Delta_2$-condition for any Orlicz function .
\end{definition}

Bikram \cite[Theorems~1.4 and~1.6]{Bikram2024} claims to extend Bufetov's
pointwise ergodic theorem for even spherical averages and the
noncommutative Rota theorem to noncommutative $L\log L$-spaces.
In particular, \cite[Theorem~1.6]{Bikram2024}
asserts that, for every factorizable trace-preserving Markov map $T$ on
a finite tracial von Neumann algebra $(\mathcal M,\tau)$, the sequence
$T^n(T^*)^n(x)$ converges b.a.u.\ for every $x\in L\log L(\mathcal M,\tau)$.
The results in \cite{Bikram2024} have been formulated for Orlicz functions $\Phi$
such that $\Phi^{\frac{1}{p}}$ is convex for some $p>1$
(\cite[Definition~3.7]{Bikram2024}). Moreover, in
\cite[Example~3.8(2)]{Bikram2024} Bikram claims that
$\Psi(t)=t\log(1+t)$ is $p$-convex for $p=10/9$. We show that this
is not the case. More generally, we show that Theorem~1.6 in \cite{Bikram2024} is completely redundant and already follows trivially from noncommutative $L_p$-Rota theorem of Anantharaman-Delaroche \cite[Corollary 4.2]{AD2006}.

\begin{proposition}\label{prop:pconvex}
Let $\Phi$ be an Orlicz function such that $g=\Phi^{\frac{1}{p}}$ is convex
for some $p>1$. Then there exist $r>0$, such that $r^p\Phi(t)\ge \Phi(r)\,t^p$ for all $t\ge r$ where $\Phi(r)>0$.
In particular, $\Psi(t)=t\log(1+t)$ is not $p$-convex for any $p>1$.
\end{proposition}

\begin{proof}
Since $g$ is convex and $g(0)=0$, for $0<s<t$ we have
$g(s)\le \tfrac{s}{t}\,g(t)$, i.e.\ $\frac{g(t)}{t}$ is nondecreasing. As $\Phi$ is unbounded then there exist some $r>0$ such that $\Phi(r)>0$ and 
taking $s=r$ gives $rg(t)\ge g(r)t$ for $t\ge r$, and raising to the
power $p$ gives the first claim. For $\Psi$ and $g_p=\Psi^{\frac{1}{p}}$,
\[
\frac{g_p(t)}{t}=t^{\frac1p-1}\bigl(\log(1+t)\bigr)^{\frac{1}{p}}\longrightarrow 0
\]
as $t\to\infty$, while $g_p(t)/t>0$. A positive nondecreasing function cannot tend to $0$. This establishes the second claim.
\end{proof}
  Now we show that the $p$-convex class in \cite{Bikram2024} does not extend noncommutative $L_p$-Rota theorem \cite[Corollary 4.2]{AD2006} beyond the class $L_p$ where $p>1.$ 
Let $\tau(1)<\infty$ and let $\Phi$ be $p$-convex for some $p>1$.
By Proposition~\ref{prop:pconvex},
for $t\geq r$, we have $t^p\leq r^p \frac{\Phi(t)}{\Phi(r)}$ and for $0\leq t<r$, $t^p\leq r^p$. Now combining those two we get
 \begin{equation}\label{eq:combine}
t^p\leq r^p+r^p \frac{\Phi(t)}{\Phi(r)},\qquad (t\geq 0)
 \end{equation}
 For $x\in L_\Phi$ and $\lambda> \norm{x}_\Phi$. We put $u=\frac{|x|}{\lambda}$, then by the definition, it gives $\tau(\Phi(u))\leq 1$.
Applying the functional calculus on \eqref{eq:combine} and taking trace, we have,
\begin{equation}\label{eq:Bi}
\frac{\norm{x}_p^p}{\lambda^p}=\tau(u^p)\leq r^p\tau(1)+r^p\frac{\tau(\Phi(u))}{\Phi(r)}\leq r^p\big(\tau(1)+\frac{1}{\Phi(r)}\big)
\end{equation}
Letting the scalar $\lambda$ decrease to $\norm{x}_\Phi$, it proves,
\[
\norm{x}_p\leq r\big(\tau(1)+\frac{1}{\Phi(r)}\big)^{\frac{1}{p}}\norm{x}_\Phi.
\]

and  $L_\Phi(\mathcal M,\tau)\subset L_p(\mathcal M,\tau)$ continuously, where $L_\Phi(\mathcal M,\tau)$ is the noncommutative Orlicz space associated with $\Phi.$
Thus, for a finite von Neumann algebra, the conclusion of \cite[Theorem~1.6]{Bikram2024}
for $L_\Phi(\mathcal M,\tau)$ is already contained in the noncommutative $L_p$-Rota theorem for $p>1$ of
\cite{AD2006}. Similarly, if $\widetilde\Phi(t)=\Phi(\sqrt t)$ is
$p$-convex, then we choose \(r>0\) with \(\Phi(r)>0\). Applying the  above proposition at \(t=s^2\) gives $\Phi(s) \geq \Phi(r)\,r^{-p}s^{2p}$  for $s\geq \sqrt{r}.$ Then $\Phi(s)\ge c\,s^{2p}$ for all sufficiently large $s$, implying
$L_\Phi\subset L_{2p}$ with $2p>2$. Thus again the corresponding a.u.\
convergence result can be deduced simply from \cite[Corollary 4.2 ]{AD2006}. Consequently, for finite von Neumann algebras, the Rota convergence conclusions in \cite[Theorem 1.6]{Bikram2024}, under its stated convexity hypotheses, follows from the corresponding noncommutative $L_p$-convergence theorem in \cite{AD2006}. 

\begin{remark}
Let $p>1.$  Also the failure occurs at the estimate used to obtain equicontinuity in
\cite[Lemma~3.9, Eq.~(7)]{Bikram2024}. For fixed $\delta>0$ the argument
requires a constant $C_\delta$ with
$\lambda\le C_\delta\,\Psi(\lambda)^{1/p}$ for all $\lambda\ge\delta/2$.
For $\Psi(\lambda)=\lambda\log(1+\lambda)$, however,
\[
\frac{\lambda}{\Psi(\lambda)^{1/p}}
=\frac{\lambda^{1-1/p}}{(\log(1+\lambda))^{1/p}}\longrightarrow\infty,
\]
so no such constant exists. Hence the $L\log L$ specialization does
not follow from \cite{Bikram2024}.
\end{remark}

\begin{remark}
 While reading the Orlicz version of Rota's theorem in \cite[Theorem~1.5]{ChattopadhyayDeShaw2026}, we observe that the authors commit similar oversight as of Bikram's \cite{Bikram2024} and we show that the result of \cite{ChattopadhyayDeShaw2026} do not reach beyond $L_p$.
By \cite[Lemma 3.2]{ChattopadhyayDeShaw2026}, it implies that
$t\mapsto\Phi(t)/t^p$ is nondecreasing. Thus, for finite traces, we apply the same argument as above \eqref{eq:Bi} and get,
$L_\Phi(\calM,\tau)\subseteq L_p(\calM,\tau)$, and the b.a.u.\
convergence follows from the noncommutative  $L_p$ Rota theorem  of \cite{AD2006}. Again applying the same argument to $\widetilde\Phi(t)=\Phi(\sqrt t)$ gives
$L_\Phi(\calM,\tau)\subseteq L_{2p}(\calM,\tau)$ under the
additional convexity assumption, so the a.u.\ conclusion follows
from the result in \cite[Corollary 4.2]{AD2006} as well. The endpoint function $\Phi(t)=t\log(e+t)$ satisfies
neither condition. The separate $L\log L$ assertion in
\cite[Theorem~3.9]{ChattopadhyayDeShaw2026} is contradicted by our
Theorem~\ref{thm:main}.
\end{remark}

\section*{Statements and Declarations}

\noindent\textbf{Competing interests:} The authors have no competing interests to declare.

\begin{sloppypar}
\textbf{Acknowledgements:} The third-named author acknowledges the support of
 the Prime Minister Early Career Research Grant
(ANRF/\allowbreak ECRG/\allowbreak 2024/\allowbreak 000699/\allowbreak PMS), and the ANRF grant
ANRF/\allowbreak ARGM/\allowbreak 2025/\allowbreak 000895/\allowbreak MTR.
\end{sloppypar}

\textbf{AI declaration:} The authors identified the failure of the $p$-convexity assumption in an earlier argument for the $L\log L$ endpoint and started investigating the endpoint problem. ChatGPT 6.1 Plus (OpenAI) was used for mathematical exploration and LaTeX editing during the preparation of the manuscript. Suggestions arising from these AI interactions were reviewed, reformulated, and independently verified by the authors before being incorporated into the manuscript. The authors take full responsibility for all mathematical content in the final manuscript.

\end{document}